\documentclass{amsart}%
\usepackage{amsmath}
\usepackage{amssymb}
\usepackage{amsfonts}
\usepackage{graphicx}%
\newtheorem{theorem}{Theorem}[section]
\theoremstyle{plain}

\newtheorem{lemma}[theorem]{Lemma}

\newtheorem{proposition}[theorem]{Proposition}

\numberwithin{equation}{section}
\begin{document}
\title[Magnetic nonlinear Choquard equation ]{Multiple solutions to a magnetic nonlinear Choquard equation}
\author{Silvia Cingolani}
\address{Dipartimento di Matematica, Politecnico di Bari, via Orabona 4, 70125 Bari, Italy.}
\email{s.cingolani@poliba.it}
\author{M\'{o}nica Clapp}
\address{Instituto de Matem\'{a}ticas, Universidad Nacional Aut\'{o}noma de M\'{e}xico,
Circuito Exterior, C.U., 04510 M\'{e}xico D.F., Mexico.}
\email{mclapp@matem.unam.mx}
\author{Simone Secchi}
\address{Dipartimento di Matematica ed Applicazioni, Universit\`{a} di Milano-Bicocca,
via Cozzi 53, 20125 Milano, Italy.}
\email{Simone.Secchi@unimib.it}
\thanks{S. Cingolani is supported by the MIUR proyect \emph{Variational and
topological methods in the study of nonlinear phenomena} (PRIN 2007).}
\thanks{M. Clapp is supported by CONACYT grant 129847 and PAPIIT grant IN101209 (Mexico).}

\begin{abstract}
We consider the stationary nonlinear magnetic Choquard equation
\[
(-\mathrm{i}\nabla+A(x))^{2}u+V(x)u=\left(  \frac{1}{|x|^{\alpha}}\ast
|u|^{p}\right)  |u|^{p-2}u,\quad x\in\mathbb{R}^{N}%
\]
where $A\ $is a real valued vector potential, $V$ is a real valued scalar
potential$,$ $N\geq3$, $\alpha\in(0,N)$ and $2-\left(  \alpha/N\right)
<p<(2N-\alpha)/(N-2)$. \ We assume that both $A$ and $V$ are compatible with
the action of some group $G$ of linear isometries of $\mathbb{R}^{N}$. We
establish the existence of multiple complex valued solutions to this equation
which satisfy the symmetry condition
\[
u(gx)=\tau(g)u(x)\text{ \ \ \ for all }g\in G,\text{ }x\in\mathbb{R}^{N},
\]
where $\tau:G\rightarrow\mathbb{S}^{1}$ is a given group homomorphism into the
unit complex numbers.\medskip


\medskip

\end{abstract}
\maketitle

\section{Introduction and statement of results}

We consider the stationary nonlinear magnetic Choquard problem
\begin{equation}%
\begin{cases}
(-\mathrm{i}\nabla+A(x))^{2}u+V(x)u=\left(  \frac{1}{|x|^{\alpha}}\ast
|u|^{p}\right)  |u|^{p-2}u,\\
u\in L^{2}(\mathbb{R}^{N},\mathbb{C}),\\
\nabla u+\mathrm{i}A(x)u\in L^{2}(\mathbb{R}^{N},\mathbb{C}^{N}),
\end{cases}
\label{prob}%
\end{equation}
where $A:\mathbb{R}^{N}\rightarrow\mathbb{R}^{N}$ is a $C^{1}$-vector
potential, $V:\mathbb{R}^{N}\rightarrow\mathbb{R}$ is a bounded continuous
scalar potential with $\inf_{\mathbb{R}^{N}}V>0$, $N\geq3,$ $\alpha\in(0,N)$
and $p\in(2-\frac{\alpha}{N},\frac{2N-\alpha}{N-2})$.

The special case
\begin{equation}
-\Delta u+u=\left(  \frac{1}{|x|}\ast|u|^{2}\right)  u,\text{\hspace{0.3in}%
}u\in H^{1}(\mathbb{R}^{3}), \label{ch}%
\end{equation}
is commonly referred to as the stationary \emph{Choquard equation}. It arises
in an approximation to Hartree-Fock theory for a one component plasma, and has
many interesting applications in the quantum theory of large systems of
non-relativistic bosonic atoms and molecules, see e.g. \cite{fl,gv,ls} and the
references therein. In his 1977 paper \cite{lieb} Lieb proved the existence
and uniqueness, up to translations, of the ground state to equation
(\ref{ch}). Later, in \cite{l2}, Lions showed the existence of a sequence of
radially symmetric solutions to this equation.

Equation (\ref{ch}) was also introduced by Penrose in his discussion on the
self-gravitational collapse of a quantum mechanical wave-function
\cite{pe1,pe2,pe3}. In this context it is usually called the
\emph{Schr\"{o}dinger-Newton equation}. Penrose suggested that the solutions
of (\ref{ch}), up to reparametrization, are the basic stationary states which
do not spontaneously collapse any further, within a certain time scale. It is
therefore of interest to investigate these basic solutions, as has been done
e.g. in \cite{mpt,mt,t}.

The eigenvalue problem associated to problems similar to (\ref{prob}) with
$N=3$ and $p=2$ has been investigated by several authors, both in the magnetic
and nonmagnetic case, see e.g. \cite{el, fl, em} and the references therein.
In \cite{a} Ackermann considered periodic potentials $V$ and proved the
existence of infinitely many solutions to problem (\ref{prob}) for $A=0$. A
problem similar to (\ref{ch}) involving a nonautonomous nonlocal term was
studied in \cite{zkhx}.

Recently, Ma and Zhao \cite{mz} studied the generalized stationary nonlinear
Choquard problem
\begin{equation}%
\begin{cases}
-\Delta u+u=\left(  \frac{1}{|x|^{\alpha}}\ast|u|^{p}\right)  |u|^{p-2}u,\\
u\in H^{1}(\mathbb{R}^{N}).
\end{cases}
\label{gch}%
\end{equation}
Under some assumptions on $\alpha,$ $p$ and $N,$ which include the classical
case, they showed that every positive solution to (\ref{gch}) is radially
symmetric and monotone decreasing about some point. Using this fact, they
proved that the positive solution to the Choquard equation (\ref{ch}) -and not
only the ground state- is unique up to translations. Uniqueness of positive
solutions in dimensions greater than $3$ is still an open question.

Semiclassical solutions to problem (\ref{prob}) for $N=3$, $\alpha=1$, $p=2$,
have been recently obtained in \cite{n,s,ww} when $A=0$, and in \cite{ccs,css}%
\ when $A\neq0$.

In this paper we consider the case where both the vector and the scalar
potential have some symmetries. To be precise, we consider a closed subgroup
$G$ of the group $O(N)$ of linear isometries of $\mathbb{R}^{N}\ $and assume
that $A$ and $V$ satisfy%
\begin{equation}
A(gx)=gA(x)\text{ \ \ and \ \ }V(gx)=V(x)\quad\text{for all $g\in G$,
$x\in\mathbb{R}^{N}$}. \label{Gassum}%
\end{equation}
We look for solutions to problem (\ref{prob}) which satisfy%
\begin{equation}
u(gx)=\tau(g)u(x)\text{ \ \ \ for all }g\in G,\text{ }x\in\mathbb{R}^{N},
\label{tau-inv}%
\end{equation}
where $\tau:G\rightarrow\mathbb{S}^{1}$ is a given continuous group
homomorphism into the unit complex numbers $\mathbb{S}^{1}.$ Solutions with
this property are called \emph{intertwining}. Condition (\ref{tau-inv})
implies that the absolute value $\left\vert u\right\vert $ of $u$ is
$G$-invariant, i.e.%
\[
\left\vert u(gx)\right\vert =\left\vert u(x)\right\vert \text{ \ \ \ for all
}g\in G,\text{ }x\in\mathbb{R}^{N},
\]
whereas the phase of $u(gx)$ is that of $u(x)$ multiplied by $\tau(g).$

It might happen that every function satisfying (\ref{tau-inv}) is trivial. For
example, if $G=O(N)$ and $\tau(g)$ is the determinant of $g$ then for each
$x\in\mathbb{R}^{N}$ we may choose a $g_{x}\in O(N)$ with $g_{x}x=x$ and
$\tau(g_{x})=-1.$ If $u:\mathbb{R}^{N}\rightarrow\mathbb{C}$ satisfies
(\ref{tau-inv}) then $u(x)=u(g_{x}x)=-u(x).$ Hence $u=0.$ To avoid this
behavior we introduce assumption $(H_{0})$ below.

First we introduce some notation. For $x\in\mathbb{R}^{N}$ we denote by
$Gx:=\{gx:g\in G\}$ the $G$-orbit of $x$ and by $G_{x}:=\{g\in G:gx=x\}$ its
isotropy group. We write $\#Gx$ for the cardinality of $Gx$ and $\ker\tau$ for
the kernel of $\tau.$ Set
\begin{align*}
\delta_{G}(x)  &  :=\left\{
\begin{array}
[c]{ll}%
\inf\{\left\vert gx-hx\right\vert :g,h\in G,\text{ }gx\neq hx\} & \text{if
}\#Gx>1,\\
2\left\vert x\right\vert  & \text{if }\#Gx=1,
\end{array}
\right. \\
\Sigma_{\tau}  &  :=\{x\in\mathbb{R}^{N}:\left\vert x\right\vert =1,\text{
}\#Gx=\min_{y\in\mathbb{R}^{N}\smallsetminus\{0\}}\#Gy,\text{ }G_{x}%
\subset\ker\tau\},\\
\delta_{\tau}  &  :=\frac{1}{2}\max_{x\in\Sigma_{\tau}}\delta_{G}(x).
\end{align*}
Then, $\delta_{\tau}\in\lbrack0,1].$ Note that $\delta_{G}(x)=0$ if
$\#Gx=\infty.$ Hence, if $\delta_{\tau}>0$, some $G$-orbit in $\Sigma_{\tau}$
must be finite.

Set%
\[
\Lambda_{\alpha,p}:=\left[  2,\frac{2N}{N-2}\right]  \cap\left(  p,\frac
{pN}{N-\alpha}\right)  \cap\left(  \frac{(2p-2)N}{N+2-\alpha},\frac
{(2p-1)N}{N+2-\alpha}\right]  \cap\left[  \frac{(2p-1)N}{2N-\alpha}%
,\infty\right)  .
\]
We prove the following results.

\begin{theorem}
\label{mainthm1}Assume that $\#Gx=\infty$ for every $x\in\mathbb{R}%
^{N}\smallsetminus\{0\}$\ and that the following holds:\newline$(H_{0})$ There
exists $x\in\mathbb{R}^{N}$ such that $G_{x}\subset\ker\tau.$\newline Then
problem \emph{(\ref{prob})} has a sequence $(u_{n})$ of nontrivial solutions
which satisfy \emph{(\ref{tau-inv})} and%
\begin{equation}
\lim_{n\rightarrow\infty}\int_{\mathbb{R}^{N}}\left(  \left\vert \nabla
u_{n}+\mathrm{i}Au_{n}\right\vert ^{2}+\left\vert u_{n}\right\vert
^{2}\right)  =\infty. \label{unbdd}%
\end{equation}

\end{theorem}

\begin{theorem}
\label{mainthm2}Assume $\Sigma_{\tau}\neq\emptyset$ and the following
hold:\newline$(H_{1})$ $\ p\in\lbrack2,(2N-\alpha)/(N-2))$ and $\Lambda
_{\alpha,p}\neq\emptyset.$\newline$(H_{2})$ $\ \lim_{\left\vert x\right\vert
\rightarrow\infty}V(x)=V_{\infty}$ and there exist $c_{0}>0,$ $\varrho>0$ and
$\kappa\in(0,2\delta_{\tau}\sqrt{V_{\infty}})$ such that%
\[
\left\vert A(x)\right\vert ^{2}+V(x)\leq V_{\infty}-c_{0}e^{-\kappa\left\vert
x\right\vert }\text{ \ \ \ for all }x\in\mathbb{R}^{N}\text{ with }\left\vert
x\right\vert \geq\varrho.
\]
Then problem \emph{(\ref{prob})} has at least one nontrivial solution which
satisfies \emph{(\ref{tau-inv})}.
\end{theorem}

Note that, if $N=3,$ $p=2$ and $\alpha=1$, then $\Lambda_{\alpha,p}=(2,9/4]$
and assumption $(H_{1})$ holds.

Theorem \ref{mainthm1}\ may be applied, in particular, to the generalized
Choquard problem (\ref{gch}). Assumption $(H_{0})$ holds for every $\tau$ if
$G$ acts freely on $\mathbb{R}^{N}\smallsetminus\{0\}.$ This is true, for
example, if $N$ is even and $G:=\mathbb{S}^{1}$ acts on $\mathbb{R}^{N}%
\equiv\mathbb{C}^{N/2}$ by complex multiplication on each complex coordinate.
If $A$ and $V$ satisfy (\ref{Gassum}) for this $\mathbb{S}^{1}$-action then,
for each $m\in\mathbb{Z},$ taking $\tau(\zeta):=\zeta^{m}$ in Theorem
\ref{mainthm1} we obtain a sequence $(u_{m,n})$ of solutions to problem
(\ref{prob}) which satisfy%
\[
u_{m,n}(\zeta x)=\zeta^{m}u_{m,n}(x)\text{ \ \ \ for all }\zeta\in
\mathbb{S}^{1},\text{ }x\in\mathbb{R}^{N}.
\]
Thus, the restriction of $u_{m,n}$ to a.e. $\mathbb{S}^{1}$-orbit in
$\mathbb{R}^{N}\smallsetminus\{0\}$ is a map of degree $m,$ for appropriately
chosen orientations. It follows that the solutions $u_{m,n},$ $m\in\mathbb{Z}%
$, $n\in\mathbb{N}$, are all different from one another. An example of a
magnetic potential $A$ satisfying (\ref{Gassum}) for this $\mathbb{S}^{1}%
$-action is $A(z_{1},...,z_{N/2})=(\mathrm{i}z_{1},...,\mathrm{i}z_{N/2}),$
whose associated magnetic field $B=$ curl$A$ is constant.

Solutions with similar properties are given by Theorem \ref{mainthm2} if $A$
and $V$ satisfy assumption\ (\ref{Gassum}) for some finite subgroup of
$\mathbb{S}^{1}$ only. Indeed, if $G$ is the cyclic group generated by
$e^{2\pi\mathrm{i}/k}$ and $(H_{1})$ and $(H_{2})$\ hold for $\delta_{\tau
}:=\sin(\pi/k),$ $k\geq2$, we obtain solutions $u_{0},\ldots,u_{k-1}$
satisfying%
\[
u_{m}(\zeta x)=\zeta^{m}u_{m}(x)\text{ \ \ \ for all }\zeta\in G,\text{ }%
x\in\mathbb{R}^{N},\text{ }m=0,...,k-1.
\]
These solutions are clearly pairwise different.

Both theorems will be proved using variational methods. The main difficulty
is, as usual, the lack of compactness. We show that the variational functional
associated to problem (\ref{prob}) subject to the constraint (\ref{tau-inv})
satisfies the Palais-Smale condition below some energy level determined by the
$G$-action and the value of $V$ at infinity. The Palais-Smale condition holds
at all levels when every $G$-orbit in $\mathbb{R}^{N}\smallsetminus\{0\}$ is
infinite. This allows us to apply the symmetric mountain pass theorem to prove
Theorem \ref{mainthm1} once we show that the domain of the variational
functional is an infinite-dimensional Hilbert space. The proof of Theorem
\ref{mainthm2}\ is based on showing that the ground state of problem
(\ref{gch}) has the proper asymptotic decay. We prove here that it does.

The paper is organized as follows: In section \ref{secvarprob} we discuss the
variational setting. In section \ref{secps} we prove the Palais-Smale
condition and Theorem \ref{mainthm1}. Section \ref{secthm2} is devoted to the
proof of Theorem \ref{mainthm2}. The required asymptotic estimates are
established in Appendix \ref{appendix}.

\section{The variational setting}

\label{secvarprob}Set $\nabla_{A}u:=\nabla u+\mathrm{i}Au,$ and consider the
real Hilbert space%
\[
H_{A}^{1}(\mathbb{R}^{N},\mathbb{C}):=\{u\in L^{2}(\mathbb{R}^{N}%
,\mathbb{C}):\nabla_{A}u\in L^{2}(\mathbb{R}^{N},\mathbb{C}^{N})\}
\]
with the scalar product
\[
\left\langle u,v\right\rangle _{A,V}:=\operatorname{Re}\int_{\mathbb{R}^{N}%
}\left(  \nabla_{A}u\cdot\overline{\nabla_{A}v}+V(x)u\overline{v}\right)  .
\]
We write
\[
\left\Vert u\right\Vert _{A,V}:=\left(  \int_{\mathbb{R}^{N}}\left(
\left\vert \nabla_{A}u\right\vert ^{2}+V(x)\left\vert u\right\vert
^{2}\right)  \right)  ^{1/2}%
\]
for the associated norm, which is equivalent to the usual one, defined by
taking $V\equiv1$ \cite[Definition 7.20]{ll}. If $u\in H_{A}^{1}%
(\mathbb{R}^{N},\mathbb{C}),$ then $\left\vert u\right\vert \in H^{1}%
(\mathbb{R}^{N})$ and%
\begin{equation}
\left\vert \nabla\left\vert u\right\vert (x)\right\vert \leq\left\vert \nabla
u(x)+\mathrm{i}A(x)u(x)\right\vert \text{ \ \ for a.e. }x\in\mathbb{R}^{N}.
\label{di}%
\end{equation}
This is called the \emph{diamagnetic inequality} \cite[Theorem 7.21]{ll}. Set
\[
\mathbb{D}(u):=\int_{\mathbb{R}^{N}}\int_{\mathbb{R}^{N}}\frac{|u(x)|^{p}%
|u(y)|^{p}}{\left\vert x-y\right\vert ^{\alpha}}\,dxdy
\]
and $r:=2N/(2N-\alpha).$ Note that $pr\in(2,2N/(N-2)).$ The classical
Hardy-Littlewood-Sobolev inequality \cite[Theorem 4.3]{ll} yields%
\begin{equation}
\left\vert \int_{\mathbb{R}^{N}}\int_{\mathbb{R}^{N}}\frac{\phi(x)\psi
(y)}{|x-y|^{\alpha}}\,dx\,dy\right\vert \leq K\Vert\phi\Vert_{L^{r}%
(\mathbb{R}^{N})}\Vert\psi\Vert_{L^{r}(\mathbb{R}^{N})}, \label{hls}%
\end{equation}
for some positive constant $K=K(\alpha,N)$ and all $\phi,\psi\in
L^{r}(\mathbb{R}^{N}).$ In particular,%
\begin{equation}
\mathbb{D}(u)\leq K\Vert u\Vert_{L^{pr}(\mathbb{R}^{N})}^{2p} \label{Dbdd}%
\end{equation}
for every $u\in H_{A}^{1}(\mathbb{R}^{N},\mathbb{C}).$ This shows that
$\mathbb{D}$ is well-defined. Inequalities (\ref{di}) and (\ref{Dbdd}),
together with Sobolev's inequality, yield%
\begin{equation}
\mathbb{D}(u)^{1/p}\leq S_{\ast}\left\Vert u\right\Vert _{A,1}^{2}
\label{magsob}%
\end{equation}
for some positive constant $S_{\ast}=S_{\ast}(p,\alpha,N)$ and every $u\in
H_{A}^{1}(\mathbb{R}^{N},\mathbb{C}).$

The energy functional $J_{A,V}:H_{A}^{1}(\mathbb{R}^{N},\mathbb{C}%
)\rightarrow\mathbb{R}$ associated to problem (\ref{prob}), defined by
\[
J_{A,V}(u):=\frac{1}{2}\left\Vert u\right\Vert _{A,V}^{2}-\frac{1}%
{2p}\mathbb{D}(u),
\]
is of class $C^{2}.$ Its derivative is given by%
\[
J_{A,V}^{\prime}(u)v:=\left\langle u,v\right\rangle _{A,V}-\operatorname{Re}%
\int_{\mathbb{R}^{N}}\left(  \frac{1}{|x|^{\alpha}}\ast|u|^{p}\right)
|u|^{p-2}u\overline{v}.
\]
Therefore, the solutions to problem (\ref{prob}) are the critical points of
$J_{A,V}.$

The action of $G$ on $H_{A}^{1}(\mathbb{R}^{N},\mathbb{C})$ given by
$(g,u)\mapsto u_{g},$ where%
\[
(u_{g})(x):=\tau(g)u(g^{-1}x),
\]
satisfies%
\[
\left\langle u_{g},v_{g}\right\rangle _{A,V}=\left\langle u,v\right\rangle
_{A,V}\text{ \ \ and \ \ }\mathbb{D}(u_{g})=\mathbb{D}(u)
\]
for all $g\in G,$ $u,v\in H_{A}^{1}(\mathbb{R}^{N},\mathbb{C}).$ Hence,
$J_{A,V}$ is $G$-invariant. By the principle of symmetric criticality \cite{p,
w}, the critical points of the restriction of $J_{A,V}$ to the fixed point
space of the $G$-action, defined as
\[%
\begin{array}
[c]{ll}%
H_{A}^{1}(\mathbb{R}^{N},\mathbb{C})^{\tau}: & =\{u\in H_{A}^{1}%
(\mathbb{R}^{N},\mathbb{C}):u_{g}=u\}\\
& =\{u\in H_{A}^{1}(\mathbb{R}^{N},\mathbb{C}):u(gx)=\tau(g)u(x)\text{
\ }\forall x\in\mathbb{R}^{N},\text{ }g\in G\},
\end{array}
\]
are the solutions to problem (\ref{prob}) which satisfy (\ref{tau-inv}). The
nontrivial ones lie on the \emph{Nehari manifold}%
\[
\mathcal{N}_{A,V}^{\tau}:=\{u\in H_{A}^{1}(\mathbb{R}^{N},\mathbb{C})^{\tau
}:u\neq0,\text{ \ }\left\Vert u\right\Vert _{A,V}^{2}=\mathbb{D}(u)\},
\]
which is radially diffeomorphic to the unit sphere in $H_{A}^{1}%
(\mathbb{R}^{N},\mathbb{C})^{\tau},$ and%
\[
E_{A,V}^{\tau}:=\inf_{u\in\mathcal{N}_{A,V}^{\tau}}J_{A,V}(u)=\inf_{u\in
H_{A}^{1}(\mathbb{R}^{N},\mathbb{C})^{\tau}\smallsetminus\{0\}}\max_{t\geq
0}J_{A,V}(u)
\]
is the first mountain pass value of the functional $J_{A,V}:H_{A}%
^{1}(\mathbb{R}^{N},\mathbb{C})^{\tau}\rightarrow\mathbb{R}$.

\section{The Palais-Smale condition}

\label{secps}Recall that $J_{A,V}:H_{A}^{1}(\mathbb{R}^{N},\mathbb{C})^{\tau
}\rightarrow\mathbb{R}$ is said to satisfy the Palais-Smale condition
$(PS)_{c}$ at the level $c$, if every sequence $(u_{n})$ such that%
\[
u_{n}\in H_{A}^{1}(\mathbb{R}^{N},\mathbb{C})^{\tau},\text{ \ \ \ \ }%
J_{A,V}(u_{n})\rightarrow c,\text{ \ \ \ \ }\nabla J_{A,V}(u_{n}%
)\rightarrow0,
\]
contains a convergent subsequence.

For $\lambda\in(0,\infty)$ we consider the problem
\begin{equation}%
\begin{cases}
-\Delta u+\lambda u=\left(  \frac{1}{|x|^{\alpha}}\ast|u|^{p}\right)
|u|^{p-2}u,\\
u\in H^{1}(\mathbb{R}^{N}),
\end{cases}
\label{lim}%
\end{equation}
and write $J_{\lambda}:H^{1}(\mathbb{R}^{N})\rightarrow\mathbb{R}$,%
\[
J_{\lambda}(u):=\frac{1}{2}\left\Vert u\right\Vert _{\lambda}^{2}-\frac{1}%
{2p}\mathbb{D}(u),\text{\hspace{0.3in}with \ \ }\left\Vert u\right\Vert
_{\lambda}^{2}:=\int_{\mathbb{R}^{N}}\left(  \left\vert \nabla u\right\vert
^{2}+\lambda u^{2}\right)  ,
\]
for its associated energy functional, and
\[
E_{\lambda}:=\inf_{u\in H^{1}(\mathbb{R}^{N})\smallsetminus\{0\}}\max_{t\geq
0}J_{\lambda}(tu)
\]
for the first mountain pass value of $J_{\lambda}$.

We shall prove the following result.

\begin{proposition}
\label{propPS}The functional $J_{A,V}:H_{A}^{1}(\mathbb{R}^{N},\mathbb{C}%
)^{\tau}\rightarrow\mathbb{R}$ satisfies $(PS)_{c}$ at each
\[
c<(\min_{x\in\mathbb{R}^{N}\setminus\{0\}}\#Gx)E_{V_{\infty}},
\]
where $V_{\infty}:=\liminf_{\left\vert x\right\vert \rightarrow\infty}V(x).$
\end{proposition}

The similar statement for the local nonmagnetic problem is well known and can
be traced back to Lions' paper \cite{l3}. The local magnetic case was recently
treated in \cite{cc1}. The proof of Proposition \ref{propPS}, though it
follows a similar pattern, requires different arguments at several points
where the facts used in \cite{cc1}\ are either not known or do not carry over
easily to the nonlocal case.

We start by recalling some basic facts about group actions, see \cite{tD} for
details. Let $x\in\mathbb{R}^{N}.$ The $G$-orbit $Gx$ of $x$ is $G$%
-homeomorphic to the homogeneous space $G/G_{x}.$ Isotropy groups satisfy
$G_{gx}=gG_{x}g^{-1}.$ Hence the set of isotropy groups of $\mathbb{R}^{N}$
includes all groups in their conjugacy classes. Note also that, if
$\#Gx<\infty,$ there is only a finite number of groups conjugate to $G_{x}.$
The conjugacy class $(G_{x})$ of an isotropy group $G_{x}$ is called an
isotropy class. The set of isotropy classes of $\mathbb{R}^{N}$ is finite.
Conjugacy classes of subgroups of $G$ are partially ordered as follows:%
\[
(K_{1})\leq(K_{2})\Longleftrightarrow\text{ there exists }g\in G\text{ such
that }gK_{1}g^{-1}\subset K_{2}.
\]
We set%
\[
(\mathbb{R}^{N})^{K}:=\{x\in\mathbb{R}^{N}:gx=x\text{ for all }g\in K\}.
\]

\begin{lemma}
\label{lemG}Let $(y_{n})$ be a sequence in $\mathbb{R}^{N}.$ Then, after
passing to a subsequence, there exist a closed subgroup $\Gamma$ of $G$ and a
sequence $(\zeta_{n})$ in $\mathbb{R}^{N}$such that\newline(a) \ $\left(
\text{\emph{dist}}(Gy_{n},\zeta_{n})\right)  $ is bounded,\newline(b)
\ $G_{\zeta_{n}}=\Gamma,$\newline(c) \ if $\left\vert G/\Gamma\right\vert
<\infty$ then $\left\vert g\zeta_{n}-\tilde{g}\zeta_{n}\right\vert
\rightarrow\infty$ for all $g,\tilde{g}\in G$ with $\tilde{g}g^{-1}%
\notin\Gamma$,\newline(d) \ if $\left\vert G/\Gamma\right\vert =\infty,$ there
exists a closed subgroup $\Gamma^{\prime}$ of $G$ such that $\Gamma
\subset\Gamma^{\prime},$ $\left\vert G/\Gamma^{\prime}\right\vert =\infty$ and
$\left\vert g\zeta_{n}-\tilde{g}\zeta_{n}\right\vert \rightarrow\infty$ for
all $g,\tilde{g}\in G$ with $\tilde{g}g^{-1}\notin\Gamma^{\prime}$.
\end{lemma}

\begin{proof}
Set $V:=\{x\in\mathbb{R}^{N}:\left\vert G/G_{x}\right\vert <\infty\}.$ Note
that $V$ is a $G$-invariant linear subspace of $\mathbb{R}^{N}.$ We consider
two cases.\newline\emph{Case 1.} The sequence $\left(  \text{dist}%
(y_{n},V)\right)  $ is bounded.\newline Let $\mathfrak{F}$ be the set of
isotropy classes $(G_{x})$ such that $x\in V$ and, for some $g\in G,$
$($dist$(y_{n},(\mathbb{R}^{N})^{gG_{x}g^{-1}}))$ contains a bounded
subsequence. We claim that $\mathfrak{F}\neq\emptyset.$ Indeed, if $z_{n}$ is
the orthogonal projection of $y_{n}$ onto $V$, after passing to a subsequence
we may assume that $G_{z_{n}}=K$ for all $n\in\mathbb{N}$. Since
$\ $dist$(y_{n},(\mathbb{R}^{N})^{K})=\left\vert y_{n}-z_{n}\right\vert =$
dist$(y_{n},V)$, we conclude that $(K)\in\mathfrak{F.}$ \newline We choose
$\Gamma$ and a subsequence of $(y_{n})$ - which we denote the same way - such
that $(\Gamma)$ is a maximal element of $\mathfrak{F}$ and
\[
\text{dist}(y_{n},(\mathbb{R}^{N})^{\Gamma})\leq c<\infty\qquad\forall
n\in\mathbb{N}\text{.}%
\]
Let $\zeta_{n}$ be the orthogonal projection of $y_{n}$ onto $(\mathbb{R}%
^{N})^{\Gamma}.$\newline\emph{(a)} is trivially satisfied since
\[
\text{dist}(Gy_{n},\zeta_{n})\leq\left\vert y_{n}-\zeta_{n}\right\vert
=(\text{dist}(y_{n},(\mathbb{R}^{N})^{\Gamma})<c\qquad\forall n\in\mathbb{N}.
\]
Passing to a subsequence we may assume that $G_{\zeta_{n}}=K$ for all
$n\in\mathbb{N}.$ Then,%
\[
\text{dist}(y_{n},(\mathbb{R}^{N})^{K})=\left\vert y_{n}-\zeta_{n}\right\vert
<c\qquad\forall n\in\mathbb{N}\text{.}%
\]
Therefore $(K)\in\mathfrak{F.}$ Since $\Gamma\subset G_{\zeta_{n}}$ and
$(\Gamma)$ is maximal, we conclude that $(G_{\zeta_{n}})=(\Gamma).$ It follows
that $G_{\zeta_{n}}=\Gamma$. This proves \emph{(b).}\newline Since $\left\vert
G/\Gamma\right\vert <\infty,$ in order to prove \emph{(c)} it suffices to show
that, if $g\notin\Gamma,$ then $(g\zeta_{n}-\zeta_{n})$ does not contain a
bounded subsequence. Arguing by contradiction, assume there exist $\hat
{g}\notin\Gamma$ and a bounded subsequence of $(\hat{g}\zeta_{n}-\zeta_{n})$.
Let $K$ be the subgroup of $G$ generated by $\Gamma\cup\{\hat{g}\},$
$W:=(\mathbb{R}^{N})^{K}$ and $W^{\perp}$ be the orthogonal complement of $W$
in $(\mathbb{R}^{N})^{\Gamma}.$ Write%
\[
\zeta_{n}=\zeta_{n}^{1}+\zeta_{n}^{2}\text{\qquad with \ }\zeta_{n}^{1}\in
W\text{ \ and }\zeta_{n}^{2}\in W^{\perp}.
\]
Then $\hat{g}\zeta_{n}^{2}-\zeta_{n}^{2}=\hat{g}\zeta_{n}-\zeta_{n}$. Since
$\hat{g}\notin\Gamma$, assertion \emph{(b)}$\ $implies that $\hat{g}\zeta
_{n}\neq\zeta_{n}.$ Hence $\zeta_{n}^{2}\neq0$ and, passing to a subsequence,
we have
\[
\frac{\zeta_{n}^{2}}{\left\vert \zeta_{n}^{2}\right\vert }\rightarrow\zeta.
\]
If $(\zeta_{n}^{2})$ is unbounded, a subsequence satisfies
\[
\left\vert \frac{\hat{g}\zeta_{n}^{2}}{\left\vert \zeta_{n}^{2}\right\vert
}-\frac{\zeta_{n}^{2}}{\left\vert \zeta_{n}^{2}\right\vert }\right\vert
=\frac{\left\vert \hat{g}\zeta_{n}-\zeta_{n}\right\vert }{\left\vert \zeta
_{n}^{2}\right\vert }\rightarrow0.
\]
Therefore $\zeta\in W,$ which is a contradiction. If, on the other hand,
$\left(  \zeta_{n}^{2}\right)  $ is bounded then, passing to a subsequence
such that $G_{\zeta_{n}^{1}}=K_{1}$ for all $n\in\mathbb{N},$ we conclude that%
\[
\text{dist}(y_{n},(\mathbb{R}^{N})^{K_{1}})=\left\vert y_{n}-\zeta_{n}%
^{1}\right\vert \leq\left\vert y_{n}-\zeta_{n}\right\vert +\left\vert
\zeta_{n}^{2}\right\vert \leq c^{\prime}<\infty.
\]
Hence $(K_{1})=(\Gamma),$ which is again a contradiction.\newline\emph{Case
2.} The sequence $\left(  \text{dist}(y_{n},V)\right)  $ is unbounded.\newline
Passing to a subsequence, we may assume that dist$(y_{n},V)\rightarrow\infty$
and that there exists an isotropy class $(\Gamma)$ such that $\left(
G_{y_{n}}\right)  =\left(  \Gamma\right)  $ for all $n\in\mathbb{N}.$ We
choose $\zeta_{n}\in Gy_{n}$ such that $G_{\zeta_{n}}=\Gamma.$ Then \emph{(a)}
and \emph{(b)}\ hold. Note that $\left\vert G/\Gamma\right\vert =\infty
$.\ \newline Let us prove \emph{(d)}.\emph{\ }Let $V^{\perp}$ be the
orthogonal complement of $V$ in $\mathbb{R}^{N}$ and $\xi_{n}$ be the
orthogonal projection of $\zeta_{n}$ onto $V^{\perp}$. Passing to a
subsequence, we have
\[
\frac{\xi_{n}}{\left\vert \xi_{n}\right\vert }\rightarrow\xi.
\]
Set $\Gamma^{\prime}:=G_{\xi}.$ Then $\Gamma\subset\Gamma^{\prime}$ and
$\left\vert G/\Gamma^{\prime}\right\vert =\infty.$ If $\left[  g\right]
,\left[  \tilde{g}\right]  \in G/\Gamma^{\prime}$ and $\left[  g\right]
\neq\left[  \tilde{g}\right]  $ we have that $d:=\left\vert g\xi-\tilde{g}%
\xi\right\vert >0.$ Let $n_{0}\in\mathbb{N}$ be such that $\left\vert
\frac{\xi_{n}}{\left\vert \xi_{n}\right\vert }-\xi\right\vert <\frac{d}{4}$
for $n\geq n_{0}.$ Then,%
\[
\frac{d}{2}\leq\left\vert g\xi-\tilde{g}\xi\right\vert -\left\vert \frac
{g\xi_{n}}{\left\vert \xi_{n}\right\vert }-g\xi\right\vert -\left\vert
\frac{\tilde{g}\xi_{n}}{\left\vert \xi_{n}\right\vert }-\tilde{g}%
\xi\right\vert \leq\left\vert \frac{g\xi_{n}}{\left\vert \xi_{n}\right\vert
}-\frac{\tilde{g}\xi_{n}}{\left\vert \xi_{n}\right\vert }\right\vert
\qquad\forall n\geq n_{0}.
\]
Consequently,
\[
\frac{d}{2}\text{dist}(\zeta_{n},V)=\frac{d}{2}\left\vert \xi_{n}\right\vert
\leq\left\vert g\xi_{n}-\tilde{g}\xi_{n}\right\vert \leq\left\vert g\zeta
_{n}-\tilde{g}\zeta_{n}\right\vert \qquad\forall n\geq n_{0}.
\]
Since dist$(\zeta_{n},V)\rightarrow\infty,$ assertion \emph{(d)} holds.
\end{proof}

Set $\Omega_{R}:=\{x\in\mathbb{R}^{N}:\left\vert x\right\vert >R\}.$ We write
\[
E_{\lambda,R}^{G}:=\inf_{u\in H_{0}^{1}(\Omega_{R})^{G}\smallsetminus
\{0\}}\max_{t\geq0}J_{\lambda}(tu)
\]
for the first mountain pass value of the restriction of $J_{\lambda}$ to the
subspace
\[
H_{0}^{1}(\Omega_{R})^{G}:=\{u\in H_{0}^{1}(\Omega_{R}):u(gx)=u(x)\text{
}\forall g\in G,x\in\Omega_{R}\}.
\]
As usual we identify a function in $H_{0}^{1}(\Omega_{R})$ with its extension
by $0$ in $H^{1}(\mathbb{R}^{N}).$

\begin{lemma}
\label{lemexterior}We have%
\[
\lim_{R\rightarrow\infty}E_{\lambda,R}^{G}=(\min_{x\in\mathbb{R}^{N}%
\setminus\{0\}}\#Gx)E_{\lambda}.
\]

\end{lemma}

\begin{proof}
It is easy to see that $E_{\lambda}\leq E_{\lambda,R}^{G}\leq(\min
_{x\in\mathbb{R}^{N}\setminus\{0\}}\#Gx)E_{\lambda}.$ Arguing by
contradiction, assume that $c:=\sup_{R\geq0}E_{\lambda,R}^{G}<(\min
_{x\in\mathbb{R}^{N}\setminus\{0\}}\#Gx)E_{\lambda}.$ Then $c\in\mathbb{R}.$
By Ekeland's variational principle \cite[Theorem 2.4]{w}, we may choose
$u_{n}\in H_{0}^{1}(\Omega_{n})^{G}$ such that
\begin{equation}
J_{\lambda}(u_{n})\rightarrow c\text{ \ \ and \ \ }\left\Vert \nabla
_{n}J_{\lambda}(u_{n})\right\Vert _{\lambda}\rightarrow0, \label{evp}%
\end{equation}
where $\nabla_{n}J_{\lambda}(u)$ denotes the orthogonal projection of $\nabla
J_{\lambda}(u)$ onto $H_{0}^{1}(\Omega_{n})^{G}.$ A standard argument yields
\[
\frac{p-1}{2p}\left\Vert u_{n}\right\Vert _{\lambda}^{2}\rightarrow c\text{
\ \ and \ \ }\frac{p-1}{2p}\mathbb{D}(u_{n})\rightarrow c.
\]
We write $B(y,s):=\{x\in\mathbb{R}^{N}:\left\vert x-y\right\vert <s\}$ and set%
\[
\eta:=\limsup_{n\rightarrow\infty}\sup_{y\in\mathbb{R}^{N}}\int_{B(y,1)}%
\left\vert u_{n}\right\vert ^{2}.
\]
Since $c>0,$ Lions' lemma \cite[Lemma 1.21]{w},\ together with inequality
(\ref{Dbdd}), yields that $\eta\neq0.$ We choose $y_{n}\in\mathbb{R}^{N}$ such
that%
\[
\int_{B(y_{n},1)}\left\vert u_{n}\right\vert ^{2}\geq\frac{\eta}{2}.
\]
Next, we replace $\left(  y_{n}\right)  $ by $\left(  \zeta_{n}\right)  $
satisfying all conditions of Lemma \ref{lemG}, and set $v_{n}(x):=u_{n}%
(x+\zeta_{n}).$ Passing to a subsequence, we may assume that $v_{n}%
\rightharpoonup v$ weakly in $H^{1}(\mathbb{R}^{N}),$ $v_{n}(x)\rightarrow
v(x)$ a.e. in $\mathbb{R}^{N},$ and $v_{n}\rightarrow v$ strongly in
$L_{loc}^{2}(\mathbb{R}^{N})$. Choosing $R>0$ such that \ dist$\left(
Gy_{n},\zeta_{n}\right)  \leq R$ for all $n\in\mathbb{N},$ we obtain%
\[
\int_{B(0,R+1)}\left\vert v_{n}\right\vert ^{2}=\int_{B(\zeta_{n}%
,R+1)}\left\vert u_{n}\right\vert ^{2}\geq\int_{B(y_{n},1)}\left\vert
u_{n}\right\vert ^{2}\geq\frac{\eta}{2}.
\]
Hence, $v\neq0$ and $\left\vert \zeta_{n}\right\vert -n\geq-(R+1)$ for $n$
large enough.\newline Note that, since $u_{n}\in H^{1}(\mathbb{R}^{N})^{G},$
we have that
\begin{equation}
u_{n}(x)=v_{n}g^{-1}(x-g\zeta_{n})\text{ \ \ for all }g\in G. \label{inv}%
\end{equation}
Let $g_{1},\ldots,g_{m}\in G$ be such that $\left\vert g_{j}\zeta_{n}%
-g_{i}\zeta_{n}\right\vert \rightarrow\infty$ if $i\neq j.$ Then,%
\[
v_{n}g_{j}^{-1}-%
{\textstyle\sum\limits_{i=j+1}^{m}}
v_{n}g_{i}^{-1}(\text{ }\cdot\text{ }-g_{i}\zeta_{n}+g_{j}\zeta_{n}%
)\rightharpoonup vg_{j}^{-1}%
\]
weakly in $H^{1}(\mathbb{R}^{N}).$ Therefore,
\begin{align*}
&  \left\Vert v_{n}g_{j}^{-1}-%
{\textstyle\sum\limits_{i=j+1}^{m}}
v_{n}g_{i}^{-1}(\text{ }\cdot\text{ }-g_{i}\zeta_{n}+g_{j}\zeta_{n}%
)\right\Vert _{\lambda}^{2}\\
&  =\left\Vert v_{n}g_{j}^{-1}-%
{\textstyle\sum\limits_{i=j}^{m}}
v_{n}g_{i}^{-1}(\text{ }\cdot\text{ }-g_{i}\zeta_{n}+g_{j}\zeta_{n}%
)\right\Vert _{\lambda}^{2}+\left\Vert vg_{j}^{-1}\right\Vert _{\lambda}%
^{2}+o(1).
\end{align*}
The change of variable $y=x-g_{j}\zeta_{n},$ together with (\ref{inv}), yields%
\[
\left\Vert u_{n}-%
{\textstyle\sum\limits_{i=j+1}^{m}}
v_{n}g_{i}^{-1}(\text{ }\cdot\text{ }-g_{i}\zeta_{n})\right\Vert _{\lambda
}^{2}=\left\Vert u_{n}-%
{\textstyle\sum\limits_{i=j}^{m}}
v_{n}g_{i}^{-1}(\text{ }\cdot\text{ }-g_{i}\zeta_{n})\right\Vert _{\lambda
}^{2}+\left\Vert v\right\Vert _{\lambda}^{2}+o(1),
\]
and iterating this equality we obtain%
\[
\left\Vert u_{n}\right\Vert _{\lambda}^{2}=\left\Vert u_{n}-%
{\textstyle\sum\limits_{i=1}^{m}}
v_{n}g_{i}^{-1}(\text{ }\cdot\text{ }-g_{i}\zeta_{n})\right\Vert _{\lambda
}^{2}+m\left\Vert v\right\Vert _{\lambda}^{2}+o(1).
\]
Multiplying by $\frac{p-1}{2p}$ and passing to the limit as $n\rightarrow
\infty$ gives%
\[
\sup_{R\geq0}E_{\lambda,R}^{G}=\lim_{n\rightarrow\infty}\frac{p-1}%
{2p}\left\Vert u_{n}\right\Vert _{\lambda}^{2}\geq\frac{p-1}{2p}m\left\Vert
v\right\Vert _{\lambda}^{2}.
\]
Since $\sup_{R\geq0}E_{\lambda,R}^{G}<\infty,$ condition \emph{(d)} in Lemma
\ref{lemG}\ implies that $\left\vert G/\Gamma\right\vert <\infty,$ where
$\Gamma=G_{\zeta_{n}}.$ Condition \emph{(c)}\ allows us to take $m=\left\vert
G/\Gamma\right\vert =\#G\zeta_{n}.$ Hence,
\[
\sup_{R\geq0}E_{\lambda,R}^{G}\geq(\min_{x\in\mathbb{R}^{N}\setminus
\{0\}}\#Gx)\frac{p-1}{2p}\left\Vert v\right\Vert _{\lambda}^{2}.
\]
To finish the proof we will show that $\frac{p-1}{2p}\left\Vert v\right\Vert
_{\lambda}^{2}\geq E_{\lambda}.$ We distinguish two cases:\newline\emph{Case
1.} If the sequence $(\left\vert \zeta_{n}\right\vert -n)$ is unbounded,
passing to a subsequence we may assume that $\left\vert \zeta_{n}\right\vert
-n\rightarrow\infty.$ Then, every compact subset of $\mathbb{R}^{N}$ is
contained in $\Omega_{n}-\zeta_{n}$ for $n$ large enough. So, by (\ref{evp}),
$v$ is a nontrivial solution to problem (\ref{lim}). Hence, $J_{\lambda
}(v)=\frac{p-1}{2p}\left\Vert v\right\Vert _{\lambda}^{2}\geq E_{\lambda}%
.$\newline\emph{Case 2.} If $(\left\vert \zeta_{n}\right\vert -n)$ is bounded
then, after passing to a subsequence, we may assume that $\left\vert \zeta
_{n}\right\vert -n\rightarrow d$ in $\mathbb{R}$ and that $\frac{\zeta_{n}%
}{\left\vert \zeta_{n}\right\vert }\rightarrow\zeta_{0}$ in $\mathbb{R}^{N}.$
Consider the half-space $\mathbb{H}:=\{x\in\mathbb{R}^{N}:(x+d\zeta_{0}%
)\cdot\zeta_{0}>0\}.$ Since every compact subset in the interior of
$\mathbb{R}^{N}\smallsetminus\mathbb{H}$ is contained in $\mathbb{R}%
^{N}\smallsetminus(\Omega_{n}-\zeta_{n})$ for large enough $n,$ we have that
$v\in H_{0}^{1}(\mathbb{H}).$ Moreover, since every compact subset of
$\mathbb{H}$ is contained in $\Omega_{n}-\zeta_{n}$ for large enough $n$, we
have that $v$ is a nontrivial solution of%
\[
\left\{
\begin{array}
[c]{l}%
-\Delta u+\lambda u=\left(  \frac{1}{|x|^{\alpha}}\ast|u|^{p}\right)
|u|^{p-2}u,\\
u\in H_{0}^{1}(\mathbb{H}).
\end{array}
\right.
\]
Hence, $J_{\lambda}(v)=\frac{p-1}{2p}\left\Vert v\right\Vert _{\lambda}%
^{2}\geq\inf_{u\in H_{0}^{1}(\mathbb{H})\smallsetminus\{0\}}\sup_{t\geq
0}J_{\lambda}(tu)\geq E_{\lambda},$ as claimed.\newline We conclude that
$\sup_{R\geq0}E_{\lambda,R}^{G}\geq(\min_{x\in\mathbb{R}^{N}\setminus
\{0\}}\#Gx)E_{\lambda},$ contradicting our assumption.
\end{proof}

\begin{lemma}
\label{lemPSto0}Let $(u_{n})$ be a sequence in $H_{A}^{1}(\mathbb{R}%
^{N},\mathbb{C})^{\tau}$ such that $u_{n}\rightharpoonup0$ weakly in
$H_{A}^{1}(\mathbb{R}^{N},\mathbb{C})$,
\[
J_{A,V}(u_{n})\rightarrow c<(\min_{x\in\mathbb{R}^{N}\setminus\{0\}}%
\#Gx)E_{V_{\infty}},\text{ \ \ and \ \ }\nabla J_{A,V}(u_{n})\rightarrow0
\]
strongly in $H_{A}^{1}(\mathbb{R}^{N},\mathbb{C}),$ where $V_{\infty}%
:=\liminf_{\left\vert x\right\vert \rightarrow\infty}V(x).$ Then a subsequence
of $(u_{n})$ converges strongly to $0$ in $H_{A}^{1}(\mathbb{R}^{N}%
,\mathbb{C}).$
\end{lemma}

\begin{proof}
A standard argument shows that%
\[
\frac{p-1}{2p}\Vert u_{n}\Vert_{A,V}^{2}\rightarrow c\text{ \ \ and \ \ }%
\frac{p-1}{2p}\mathbb{D}(u_{n})\rightarrow c.
\]
So, if $c\leq0,$ then $u_{n}\rightarrow0$ strongly in $H_{A}^{1}%
(\mathbb{R}^{N},\mathbb{C})$ and we are done. \newline Assume that $c>0.$ Fix
$\lambda<V_{\infty}$ such that $c<(\min_{x\in\mathbb{R}^{N}\setminus
\{0\}}\#Gx)E_{\lambda},$ and $R_{0}>0$ such that $V(x)\geq\lambda$ if
$\left\vert x\right\vert \geq R_{0}.$ Let $\varepsilon\in(0,1).$ Since
$(u_{n})$ is bounded in $H_{A}^{1}(\mathbb{R}^{N},\mathbb{C})$ there exists
$R_{\varepsilon}>R_{0}$ such that $R_{\varepsilon}\rightarrow\infty$ as
$\varepsilon\rightarrow0$ and, after passing to a subsequence,%
\[
\int_{R_{\varepsilon}<\left\vert x\right\vert <R_{\varepsilon}+1}\left(
\left\vert \nabla_{A}u_{n}\right\vert ^{2}+V(x)\left\vert u_{n}\right\vert
^{2}\right)  <\varepsilon\text{ \ \ for all }n\in\mathbb{N}.
\]
We may assume that $\left\vert u_{n}\right\vert \rightarrow0$ strongly in
$L_{loc}^{rp}(\mathbb{R}^{N})$ with $r:=2N/(2N-\alpha)$. Let $\chi\in
C^{\infty}(\mathbb{R}^{N},\mathbb{R})$ be radial and such that $\chi(x)=0$ if
$\left\vert x\right\vert \leq R_{\varepsilon}$, $\chi(x)=1$ if $\left\vert
x\right\vert \geq R_{\varepsilon}+1,$ and $\chi(x)\in\lbrack0,1]$ for all
$x\in\mathbb{R}^{N}.$ Set $w_{n}:=\chi u_{n}.$ Then, using (\ref{hls}) we
obtain%
\begin{align}
\left\vert \mathbb{D}(u_{n})-\mathbb{D}(w_{n})\right\vert  &  \leq
\int_{\mathbb{R}^{N}}\int_{\mathbb{R}^{N}}\frac{\left\vert |u_{n}%
(x)|^{p}|u_{n}(y)|^{p}-|w_{n}(x)|^{p}|w_{n}(y)|^{p}\right\vert }{\left\vert
x-y\right\vert ^{\alpha}}\,dxdy\nonumber\\
&  \leq2\int_{\mathbb{R}^{N}}\int_{\mathbb{R}^{N}}\frac{|u_{n}(x)|^{p}%
\left\vert |u_{n}(y)|^{p}-|w_{n}(y)|^{p}\right\vert }{\left\vert
x-y\right\vert ^{\alpha}}\,dxdy\nonumber\\
&  \leq2K\Vert u_{n}\Vert_{L^{pr}(\mathbb{R}^{N})}^{p}\Vert|u_{n}%
(x)|^{p}-|w_{n}(y)|^{p}\Vert_{L^{r}(\mathbb{R}^{N})}\nonumber\\
&  \leq C\left\Vert u_{n}\right\Vert _{L^{rp}(B(0,R_{\varepsilon}+1))}%
^{p}=o(1). \label{difD}%
\end{align}
Here and in the following $C$ denotes some positive constant independent of
$n$, not necessarily the same one. Similarly,%
\begin{align*}
&  \left\vert \int_{\mathbb{R}^{N}}\left(  \frac{1}{\left\vert x\right\vert
^{\alpha}}\ast\left\vert u_{n}\right\vert ^{p}\right)  \left\vert
u_{n}\right\vert ^{p-2}u_{n}\bar{w}_{n}-\int_{\mathbb{R}^{N}}\left(  \frac
{1}{\left\vert x\right\vert ^{\alpha}}\ast\left\vert w_{n}\right\vert
^{p}\right)  \left\vert w_{n}\right\vert ^{p-2}w_{n}\bar{w}_{n}\right\vert \\
&  \leq C\left\Vert u_{n}\right\Vert _{L^{rp}(B(0,R_{\varepsilon}+1))}%
^{p}=o(1).
\end{align*}
Therefore,%
\begin{align*}
\left\vert J_{A,V}^{\prime}(u_{n})w_{n}-J_{A,V}^{\prime}(w_{n})w_{n}%
\right\vert  &  \leq\left\vert \left\langle u_{n}-w_{n},w_{n}\right\rangle
_{A,V}\right\vert +o(1)\\
&  \leq C\int_{R_{\varepsilon}<\left\vert x\right\vert <R_{\varepsilon}%
+1}\left(  \left\vert \nabla_{A}u_{n}\right\vert ^{2}+V(x)\left\vert
u_{n}\right\vert ^{2}\right)  +o(1).
\end{align*}
Since $J_{A,V}^{\prime}(u_{n})w_{n}\rightarrow0$ we conclude that%
\[
\left\vert \left\Vert w_{n}\right\Vert _{A,V}^{2}-\mathbb{D}(w_{n})\right\vert
<C\varepsilon\text{ \ \ for }n\text{ large enough.}%
\]
Noting that $\mathbb{D}(w_{n})\rightarrow\left(  2p/(p-1)\right)  c>0$ and
using the diamagnetic inequality (\ref{di}) we obtain that
\begin{equation}
\frac{\left\Vert \left\vert w_{n}\right\vert \right\Vert _{\lambda}^{2}%
}{\mathbb{D}(w_{n})^{1/p}}\leq\frac{\left\Vert w_{n}\right\Vert _{A,V}^{2}%
}{\mathbb{D}(w_{n})^{1/p}}\leq\mathbb{D}(w_{n})^{(p-1)/p}+C\varepsilon
\label{comp}%
\end{equation}
for $n$ large enough. Observe that%
\[
c_{\lambda,R_{\varepsilon}}^{G}\leq\max_{t\geq0}J_{A,V}(t\left\vert
w_{n}\right\vert )=\frac{p-1}{2p}\left(  \frac{\left\Vert \left\vert
w_{n}\right\vert \right\Vert _{\lambda}^{2}}{\mathbb{D}(w_{n})^{1/p}}\right)
^{p/(p-1)}.
\]
Together with (\ref{comp}), this inequality yields
\[
c_{\lambda,R_{\varepsilon}}^{G}\leq\frac{p-1}{2p}\left(  \mathbb{D}%
(w_{n})^{(p-1)/p}+C\varepsilon\right)  ^{p/(p-1)},
\]
So letting first $n\rightarrow\infty$ and then $\varepsilon\rightarrow0$ we
conclude that
\[
\lim_{R\rightarrow\infty}c_{\lambda,R}^{G}\leq c<(\min_{x\in\mathbb{R}%
^{N}\setminus\{0\}}\#Gx)E_{\lambda}.
\]
This contradicts Lemma \ref{lemexterior}. Hence, $c=0$ and $u_{n}\rightarrow0$
strongly in $H_{A}^{1}(\mathbb{R}^{N},\mathbb{C}).$
\end{proof}

\begin{lemma}
\label{lemnils} Let $(u_{n})$ be a sequence in $H_{A}^{1}(\mathbb{R}%
^{N},\mathbb{C})^{\tau}$ such that $u_{n}\rightharpoonup u$ weakly in
$H_{A}^{1}(\mathbb{R}^{N},\mathbb{C}).$ The following hold:\newline(i)
\ $\mathbb{D}^{\prime}(u_{n})v\rightarrow\mathbb{D}^{\prime}(u)v$ \ for all
$v\in H_{A}^{1}(\mathbb{R}^{N},\mathbb{C}).$\newline(ii) After passing to a
subsequence, there exists a sequence $(\widetilde{u}_{n})$ in $H_{A}%
^{1}(\mathbb{R}^{N},\mathbb{C})^{\tau}$ such that $\widetilde{u}%
_{n}\rightarrow u$ strongly in $H_{A}^{1}(\mathbb{R}^{N},\mathbb{C})$,%
\begin{align*}
\mathbb{D}(u_{n})-\mathbb{D}(u_{n}-\widetilde{u}_{n})  &  \rightarrow
\mathbb{D}(u)\text{ \ \ in }\mathbb{R},\\
\mathbb{D}^{\prime}(u_{n})-\mathbb{D}^{\prime}(u_{n}-\widetilde{u}_{n})  &
\rightarrow\mathbb{D}^{\prime}(u)\text{ \ \ in }\left[  H_{A}^{1}%
(\mathbb{R}^{N},\mathbb{C})\right]  ^{\ast}.
\end{align*}

\end{lemma}

\begin{proof}
The proof is completely analogous to that of Lemma 3.5 in \cite{a}. The
function $\widetilde{u}_{n}$ is the product of $u_{n}$ with a radial cut-off
function, so $\widetilde{u}_{n}$ belongs to $H_{A}^{1}(\mathbb{R}%
^{N},\mathbb{C})^{\tau}$\ if $u_{n}$ does. We omit the details.
\end{proof}

\bigskip

\noindent\textbf{Proof of Proposition \ref{propPS}.}\qquad Let $u_{n}\in
H_{A}^{1}(\mathbb{R}^{N},\mathbb{C})^{\tau}$ satisfy
\[
J_{A,V}(u_{n})\rightarrow c<(\min_{x\in\mathbb{R}^{N}\setminus\{0\}}%
\#Gx)E_{V_{\infty}},\text{ \ \ and \ \ }\nabla J_{A,V}(u_{n})\rightarrow0
\]
strongly in $H_{A}^{1}(\mathbb{R}^{N},\mathbb{C}).$ Since $(u_{n})$ is bounded
in $H_{A}^{1}(\mathbb{R}^{N},\mathbb{C})$ it contains a subsequence such that
$u_{n}\rightharpoonup u$ weakly in $H_{A}^{1}(\mathbb{R}^{N},\mathbb{C}%
)^{\tau}.$ By Lemma \ref{lemnils}, $u$ solves (\ref{prob}) and, after passing
to a subsequence, there exists a sequence $(\widetilde{u}_{n})$ in $H_{A}%
^{1}(\mathbb{R}^{N},\mathbb{C})^{\tau}$ such that $v_{n}:=u_{n}-\widetilde
{u}_{n}\rightharpoonup0$ weakly in $H_{A}^{1}(\mathbb{R}^{N},\mathbb{C})$,%
\begin{align*}
J_{A,V}(u_{n})-J_{A,V}(v_{n})  &  \rightarrow J_{A,V}(u)\text{ \ \ \ in
}\mathbb{R},\\
\nabla J_{A,V}(u_{n})-\nabla J_{A,V}(v_{n})  &  \rightarrow0\text{
\ \ \ strongly in }H_{A}^{1}(\mathbb{R}^{N},\mathbb{C}).
\end{align*}
Hence, $J_{A,V}(u)\geq0,$
\[
J_{A,V}(v_{n})\rightarrow c-J_{A,V}(u)\leq c,\text{ \ \ and \ \ \ }\nabla
J_{A,V}(v_{n})\rightarrow0
\]
strongly in $H_{A}^{1}(\mathbb{R}^{N},\mathbb{C}).$ By Lemma \ref{lemPSto0}\ a
subsequence of $(v_{n})$ converges strongly to $0$ in $H_{A}^{1}%
(\mathbb{R}^{N},\mathbb{C}).$ This implies that a subsequence of $\left(
u_{n}\right)  $ converges strongly to $u$ in $H_{A}^{1}(\mathbb{R}%
^{N},\mathbb{C}).$ \qed\noindent

\bigskip

There exists a unique isotropy class $(P)$ such that
\[
\mathbb{R}_{(P)}^{N}:=\{x\in\mathbb{R}^{N}:(G_{x})=(P)\}
\]
is open and dense in $\mathbb{R}^{N}.$ $(P)$ is called the \emph{principal
isotropy class. }Any other isotropy class satisfies $(G_{x})\geq(P)$, see e.g.
\cite[Chapter I, Theorem 5.14]{tD}.

\bigskip

\noindent\textbf{Proof of Theorem \ref{mainthm1}.}\qquad By Proposition
\ref{propPS}, if $\#Gx=\infty$ for every $x\in\mathbb{R}^{N}\smallsetminus
\{0\}$ then $J_{A,V}:H_{A}^{1}(\mathbb{R}^{N},\mathbb{C})^{\tau}%
\rightarrow\mathbb{R}$ satisfies $(PS)_{c}$ for every $c\in\mathbb{R}.$ The
group $\mathbb{S}^{1}$ acts on $H_{A}^{1}(\mathbb{R}^{N},\mathbb{C})^{\tau}$
by complex multiplication, with $0$ as its only fixed point, and $J_{A,V}$ is
$\mathbb{S}^{1}$-invariant. Standard arguments, using inequality
(\ref{magsob}), show that $J_{A,V}$ satisfies the mountain pass conditions of
the $\mathbb{S}^{1}$-symmetric mountain pass theorem, cf. \cite[Theorem
3.14]{fhr} or \cite[Theorem 1.5]{cp}. In order to conclude that $J_{A,V}$ has
an unbounded sequence of critical values we need to show that $H_{A}%
^{1}(\mathbb{R}^{N},\mathbb{C})^{\tau}$ is infinite dimensional. \newline The
quotient map $\mathbb{R}_{(P)}^{N}\rightarrow\mathbb{R}_{(P)}^{N}/G$ onto the
$G$-orbit space of $\mathbb{R}_{(P)}^{N}$ is a smooth fibre bundle with fibre
$G/P.$ Hence, the $G$-orbit of every $x\in\mathbb{R}_{(P)}^{N}$ has a tubular
neighborhood $U_{x}$ which is $G$-diffeomorphic to the product $B\times\left(
G/P\right)  $ of an open ball $B$ of dimension $N-\dim G/P$ with $G/P,$ with
the obvious $G$-action. Assumption $(H_{0})$ implies that $P\subset\ker\tau.$
Hence, the function%
\[
C_{c}^{\infty}(B,\mathbb{C})\rightarrow C_{c}^{\infty}(B\times\left(
G/P\right)  ,\mathbb{C})^{\tau}\cong C_{c}^{\infty}(U_{x},\mathbb{C})^{\tau},
\]
given by $\varphi\mapsto\tilde{\varphi}$ where $\tilde{\varphi}(z,gP):=\tau
(g)\varphi(z),$ is well defined and it is a linear isomorphism. Here the
superindex $\tau$ indicates again the subspaces of functions satisfying
(\ref{tau-inv}). Since $C_{c}^{\infty}(U_{x},\mathbb{C})^{\tau}\subset
H_{A}^{1}(\mathbb{R}^{N},\mathbb{C})^{\tau}$ and $C_{c}^{\infty}%
(B,\mathbb{C})$\ is infinite dimensional, it follows that $H_{A}%
^{1}(\mathbb{R}^{N},\mathbb{C})^{\tau}$ is infinite dimensional. \qed\noindent

\section{Proof of Theorem \ref{mainthm2}}

\label{secthm2}Throughout this section we assume that $(H_{1})$ and $(H_{2})$ hold.

It is known that the first mountain pass value $E_{\lambda}$ of the functional
$J_{\lambda}$ associated to problem (\ref{lim}) is attained at a positive
function $\omega_{\lambda}\in H^{1}(\mathbb{R}^{N}).$ This can be proved using
Lions' concentration compactness method \cite{l}. Recently, Ma and Zhao showed
that, if $(H_{1})$ holds, then every positive solution of (\ref{lim}) is
radially symmetric \cite{mz}. After translation, we may assume that
$\omega_{\lambda}$ is radially symmetric with respect to the origin. Moreover,
for every $\mu\in(0,\lambda),$%
\begin{equation}
\left.
\begin{array}
[c]{r}%
\omega_{\lambda}(x)=O(\left\vert x\right\vert ^{-\frac{N-1}{2}}e^{-\sqrt{\mu
}\left\vert x\right\vert })\\
\left\vert \nabla\omega_{\lambda}(x)\right\vert =O(\left\vert x\right\vert
^{-\frac{N-1}{2}}e^{-\sqrt{\mu}\left\vert x\right\vert })
\end{array}
\right\}  \text{ \ \ as }\left\vert x\right\vert \rightarrow\infty.
\label{asymp}%
\end{equation}
We prove this in Appendix \ref{appendix}. We write $\omega_{\infty}%
:=\omega_{V_{\infty}}$ for $V_{\infty}:=\lim_{\left\vert x\right\vert
\rightarrow\infty}V(x)$.

Let $\kappa$ be as in assumption $(H_{2}).$ Fix $\mu\in(0,V_{\infty})$ such
that $\kappa\in(0,2\delta_{\tau}\sqrt{\mu}).$ Fix $\varepsilon\in
(0,(2\delta_{\tau}\sqrt{\mu}-\kappa)/(2\delta_{\tau}\sqrt{\mu}+\kappa))$ and a
nonincreasing cut-off function $\chi\in C^{\infty}[0,\infty)$ such that
$\chi(t)=1$ if $t\leq1-\varepsilon$ and $\chi(t)=0$ if $t\geq1.$ For
$R\in(0,\infty)$ and $u\in H^{1}(\mathbb{R}^{N})$ we write%
\[
u^{R}(x):=\chi\left(  \frac{\left\vert x\right\vert }{R}\right)  u(x).
\]
In the following $C$ will denote some positive constant, not necessarily the
same one.

\begin{lemma}
\label{lemassymp}As $R\rightarrow\infty,$%
\begin{align}
\int_{\mathbb{R}^{N}}\left\vert \left\vert \nabla\omega_{\infty}\right\vert
^{2}-\left\vert \nabla\omega_{\infty}^{R}\right\vert ^{2}\right\vert  &
=O(e^{-2\sqrt{\mu}(1-\varepsilon)R}),\label{b}\\
\left\vert \mathbb{D}(\omega_{\infty})-\mathbb{D}(\omega_{\infty}%
^{R})\right\vert  &  =O(e^{-p\sqrt{\mu}(1-\varepsilon)R}). \label{c}%
\end{align}

\end{lemma}

\begin{proof}
To prove (\ref{c}) we apply (\ref{hls}) as in (\ref{difD}), and (\ref{asymp})
to obtain
\begin{align*}
\left\vert \mathbb{D}(\omega_{\infty})-\mathbb{D}(\omega_{\infty}%
^{R})\right\vert  &  \leq2K\Vert\omega_{\infty}^{p}\Vert_{L^{r}(\mathbb{R}%
^{N})}\Vert\omega_{\infty}^{p}-(\omega_{\infty}^{R})^{p}\Vert_{L^{r}%
(\mathbb{R}^{N})}\\
&  \leq C\left(  \int_{\left\vert x\right\vert \geq(1-\varepsilon)R}%
\omega_{\infty}^{pr}(x)dx\right)  ^{1/r}\\
&  \leq C\left(  \int_{(1-\varepsilon)R}^{\infty}e^{-pr\sqrt{\mu}t}dt\right)
^{1/r}=Ce^{-p\sqrt{\mu}(1-\varepsilon)R},
\end{align*}
as $R\rightarrow\infty$. The proof of (\ref{b}) is also easy.
\end{proof}

For $y\in\mathbb{R}^{N}$ set $R_{y}:=\left[  (\kappa+2\delta_{\tau}\sqrt{\mu
})/4\delta_{\tau}\sqrt{\mu}\right]  \delta_{\tau}\left\vert y\right\vert $.
Note that $R_{y}\in(0,\delta_{\tau}\left\vert y\right\vert ).$ Since
$\delta_{\tau}\in(0,1]$ we have that $\left\vert y\right\vert -R_{y}%
\rightarrow\infty$ as $\left\vert y\right\vert \rightarrow\infty.$ For $u\in
H^{1}(\mathbb{R}^{N})$ we write $u_{y}(x):=u(x-y).$

\begin{lemma}
\label{lemub}There exist $\varrho_{0},d_{0}\in(0,\infty)$ such that%
\[
J_{A,V}(t(\omega_{\infty}^{R_{y}})_{y})\leq E_{V_{\infty}}-d_{0}%
e^{-\kappa\left\vert y\right\vert }\text{ \ \ \ for all }t\geq0\text{ \ if
}\left\vert y\right\vert \geq\varrho_{0}.
\]

\end{lemma}

\begin{proof}
Let $u\in H^{1}(\mathbb{R}^{N}).$ Observe that, since $u$ is real-valued,
$\left\vert \nabla_{A}u\right\vert ^{2}=\left\vert \nabla u\right\vert
^{2}+\left\vert A\right\vert ^{2}u^{2}$. Therefore, $(H_{2})$ implies there
exist positive constants $C_{1},C_{2}$ such that
\[
C_{1}\left\Vert u\right\Vert _{V_{\infty}}^{2}\leq\left\Vert u\right\Vert
_{A,V}^{2}\leq C_{2}\left\Vert u\right\Vert _{V_{\infty}}^{2}\text{ \ \ for
all }u\in H^{1}(\mathbb{R}^{N}).
\]
Note also that%
\[
\max_{t\geq0}J_{A,V}(tu)=J_{A,V}(t_{u}u)\text{ \ \ iff \ \ }t_{u}=\left(
\frac{\left\Vert u\right\Vert _{A,V}^{2}}{\mathbb{D}(u)}\right)  ^{1/(2p-2)}.
\]
So, since $\omega_{\infty}^{R_{y}}\rightarrow\omega_{\infty}$ in
$H^{1}(\mathbb{R}^{N})$ as $\left\vert y\right\vert \rightarrow\infty,$ there
exist $0<t_{1}<t_{2}<\infty$ such that
\[
\max_{t\geq0}J_{A,V}(t(\omega_{\infty}^{R_{y}})_{y})=\max_{t_{1}\leq t\leq
t_{2}}J_{A,V}(t(\omega_{\infty}^{R_{y}})_{y})
\]
for all large enough $\left\vert y\right\vert $. \newline Let $t\in\lbrack
t_{1},t_{2}].$ Assumption $(H_{2})$ yields%
\begin{align*}
\int_{\mathbb{R}^{N}}(\left\vert A\right\vert ^{2}+V)(t\omega_{\infty}^{R_{y}%
})_{y}^{2}  &  \leq t^{2}\int_{\left\vert x\right\vert \leq R_{y}}(\left\vert
A(x+y)\right\vert ^{2}+V(x+y))(\omega_{\infty}^{R_{y}})^{2}(x)dx\\
&  \leq t^{2}\int_{\left\vert x\right\vert \leq R_{y}}(V_{\infty}%
-c_{0}e^{-\kappa\left\vert x+y\right\vert })\omega_{\infty}^{2}(x)dx\\
&  \leq\int_{\mathbb{R}^{N}}V_{\infty}\left(  t\omega_{\infty}\right)
^{2}-\left(  c_{0}t_{1}^{2}\int_{\left\vert x\right\vert \leq1}e^{-\kappa
\left\vert x\right\vert }\omega_{\infty}^{2}(x)dx\right)  e^{-\kappa\left\vert
y\right\vert }%
\end{align*}
for $\left\vert y\right\vert $ large enough. Therefore, using Lemma
\ref{lemassymp}, we get%
\begin{align*}
J_{A,V}(t(\omega_{\infty}^{R_{y}})_{y})  &  =\frac{1}{2}\left\Vert
t(\omega_{\infty}^{R_{y}})_{y}\right\Vert _{A,V}^{2}-\frac{1}{2p}%
\mathbb{D}(t(\omega_{\infty}^{R_{y}})_{y})\\
&  \leq\frac{1}{2}\left\Vert t\omega_{\infty}\right\Vert _{V_{\infty}}%
^{2}-\frac{1}{2p}\mathbb{D}(t\omega_{\infty})-Ce^{-\kappa\left\vert
y\right\vert }+O(e^{-2\sqrt{\mu}(1-\varepsilon)R_{y}})\\
&  \leq\max_{t\geq0}J_{V_{\infty}}(t\omega_{\infty})-d_{0}e^{-\kappa\left\vert
y\right\vert }\\
&  =E_{V_{\infty}}-d_{0}e^{-\kappa\left\vert y\right\vert }%
\end{align*}
for sufficiently large $\left\vert y\right\vert $, because our choices of
$\varepsilon$ and $R_{y}$ guarantee that $2\sqrt{\mu}(1-\varepsilon
)R_{y}>\kappa\left\vert y\right\vert .$
\end{proof}

\bigskip

\noindent\textbf{Proof of Theorem \ref{mainthm2}.}\qquad Choose $\xi
\in\mathbb{R}^{N}$ such that $\#G\xi=\min_{x\in\mathbb{R}^{N}\smallsetminus
\{0\}}\#Gx,$\ $\left\vert \xi\right\vert =1,\ G_{\xi}\subset\ker\tau,$ and
\[
\min\{\left\vert g\xi-h\xi\right\vert :g,h\in G,\text{ }g\xi\neq
h\xi\}=2\delta_{\tau}\text{ \ \ if }\#G\xi>1.
\]
Note that $\#G\xi<\infty$ because $\delta_{\tau}>0.$ Set $y:=\varrho_{0}\xi,$
with $\varrho_{0}$ as in Lemma \ref{lemub}, and%
\[
\theta:=%
{\textstyle\sum\limits_{gy\in Gy}}
\tau(g)(\omega_{\infty}^{R_{y}})_{gy}.
\]
Since $\omega_{\infty}$ is radially symmetric we have that $\theta
(gx)=\tau(g)\theta(x)$ for every $g\in G,$ $x\in\mathbb{R}^{N}.$ Moreover,
since $R_{y}<\delta_{\tau}\left\vert y\right\vert ,$ the functions
$(\omega_{\infty}^{R_{y}})_{gy}$ and $(\omega_{\infty}^{R_{y}})_{hy}$ have
disjoint supports if $g\xi\neq h\xi.$ So applying Lemma \ref{lemub} we obtain%
\[
J_{A,V}(t\theta)=\left(  \#G\xi\right)  J_{A,V}(t(\omega_{\infty}^{R_{y}}%
)_{y})<\left(  \#G\xi\right)  E_{V_{\infty}}%
\]
for all $t\geq0.$ This implies that $E_{A,V}^{\tau}<(\min_{x\in\mathbb{R}%
^{N}\smallsetminus\{0\}}\#Gx)E_{V_{\infty}}.$ By Proposition \ref{propPS}%
,\ $J_{A,V}$ satisfies $(PS)_{c}$ at $c=E_{A,V}^{\tau}.$ The classical
mountain pass theorem of Ambrosetti and Rabinowitz \cite{ar} yields a solution
$u$ of problem (\ref{prob}) which satisfies (\ref{tau-inv}) and $J_{A,V}%
(u)=E_{A,V}^{\tau}.$ \qed\noindent

\appendix

\section{Asymptotic decay of ground states}

\label{appendix}We shall prove (\ref{asymp}). Set $K(x):=\frac{1}{\left\vert
x\right\vert ^{\alpha}}.$ The following lemma highlights the role of the
assumption $\Lambda_{\alpha,p}\neq\emptyset$.

\begin{lemma}
\label{conv}If $(H_{1})$ holds, then every solution $u\in H^{1}(\mathbb{R}%
^{N})$ to problem \emph{(\ref{lim})} has the following properties:\newline(i)
\ $\ u\in L^{r}(\mathbb{R}^{N})$ for every $r\in\lbrack2,\infty).$\newline(ii)
\ There exist $p<p_{1}\leq p_{2}$ and $C>0$ such that%
\[
\left\vert \left(  K\ast\left\vert u\right\vert ^{p}\right)  (x)\right\vert
\leq C\left(  \left\Vert u\right\Vert _{L^{p_{1}}(\mathbb{R}^{N})}%
^{p}+\left\Vert u\right\Vert _{L^{p_{2}}(\mathbb{R}^{N})}^{p}\right)  \text{
\ \ for all }x\in\mathbb{R}^{N}.
\]
(iii) $K\ast\left\vert u\right\vert ^{p}$ is continuous on $\mathbb{R}^{N}$
and $\lim_{\left\vert x\right\vert \rightarrow\infty}\left(  K\ast\left\vert
u\right\vert ^{p}\right)  (x)=0.$\newline(iv) $\ u$ is of class $C^{2}.$
\end{lemma}

\begin{proof}
We prove \emph{(i)} by a bootstrapping argument. Fix $q\in\Lambda_{\alpha,p}$
and set%
\[
\frac{1}{t}:=\frac{p}{q}-\frac{N-\alpha}{N}.
\]
Since $q\in\Lambda_{\alpha,p}$ we have that $q\in(p,2N/(N-2)]$ and
$1/t\in(0,1).$ Hence $K\ast\left\vert u\right\vert ^{p}\in L^{t}%
(\mathbb{R}^{N}),$ cf. \cite[Section 4.3 (9)]{ll}. Set $r_{0}:=qN/(N+2q)$ and%
\[
\frac{1}{r_{1}}:=\frac{(p-1)(N-2r_{0})}{Nr_{0}}+\frac{1}{t}=\frac{p-1}%
{q}+\frac{1}{t}.
\]
Since $q\in\Lambda_{\alpha,p}$ we have that $1/r_{1}<1/r_{0}$ and $1/r_{1}%
\in(0,1]$ so, by H\"{o}lder's inequality, $\left(  K\ast\left\vert
u\right\vert ^{p}\right)  \left\vert u\right\vert ^{p-2}u\in L^{r_{1}%
}(\mathbb{R}^{N}).$ Using $L^{q}$-regularity theory we conclude that $u\in
W^{2,r_{1}}(\mathbb{R}^{N}).$ Sobolev's inequality then implies that $u\in
L^{r}(\mathbb{R}^{N})$ for every $r\in\lbrack2,Nr_{1}/(N-2r_{1})]$ if
$2r_{1}<N,$ and $u\in L^{r}(\mathbb{R}^{N})$ for every $r\in\lbrack2,\infty)$
if $2r_{1}\geq N.$ If $2r_{k}<N$ we continue the process by setting%
\[
\frac{1}{r_{k+1}}:=\frac{(p-1)(N-2r_{k-1})}{Nr_{k-1}}+\frac{1}{t}.
\]
Note that%
\[
\frac{1}{r_{k}}-\frac{1}{r_{k+1}}=(p-1)\left[  \frac{1}{r_{k-1}}-\frac
{1}{r_{k}}\right]  =(p-1)^{k}\left[  \frac{1}{r_{0}}-\frac{1}{r_{1}}\right]
.
\]
Hence $1/r_{k}\geq2/N$ for some $k\geq1$ and arguing as above we conclude that
$u\in L^{r}(\mathbb{R}^{N})$ for every $r\in\lbrack2,\infty)$. \newline To
prove \emph{(ii)} we fix $\delta\in(0,N-\alpha)$ and set $s:=(N-\delta
)/\alpha$ and $t:=(N+\delta)/\alpha.$ Write $K=K_{1}+K_{2}$ with $K_{1}\in
L^{s}(\mathbb{R}^{N})$ and $K_{2}\in L^{t}(\mathbb{R}^{N}).$ Since $p\geq2,$
assertion \emph{(i)} implies that $\left\vert u\right\vert ^{p}\in
L^{r}(\mathbb{R}^{N})$ for every $r\in\lbrack1,\infty).$ So using H\"{o}lder's
inequality we obtain%
\begin{align*}
\left\vert \left(  K_{1}\ast\left\vert u\right\vert ^{p}\right)
(x)\right\vert  &  \leq\left\Vert K_{1}\right\Vert _{L^{s}(\mathbb{R}^{N}%
)}\left\Vert u\right\Vert _{L^{ps^{\prime}}(\mathbb{R}^{N})}^{p},\\
\left\vert \left(  K_{2}\ast\left\vert u\right\vert ^{p}\right)
(x)\right\vert  &  \leq\left\Vert K_{2}\right\Vert _{L^{t}(\mathbb{R}^{N}%
)}\left\Vert u\right\Vert _{L^{pt^{\prime}}(\mathbb{R}^{N})}^{p},
\end{align*}
which immediately yield \emph{(ii)}. Moreover, applying \cite[Lemma 2.20]{ll}
we obtain that $K_{j}\ast\left\vert u\right\vert ^{p}$ is continuous and%
\[
\lim_{\left\vert x\right\vert \rightarrow\infty}\left(  K_{j}\ast\left\vert
u\right\vert ^{p}\right)  (x)=0,\text{ \ \ \ }j=1,2.
\]
This proves \emph{(iii)}. Property \emph{(iv)} follows from assertion
\emph{(i),} Sobolev's embedding theorem and Schauder theory in the usual way,
cf. e.g. \cite[Appendix B]{st}.
\end{proof}

Our set $\Lambda_{\alpha,p}$ coincides with the set $\Lambda$ defined by Ma
and Zhao in \cite[Remark 3]{mz}, except possibly for the point
$2(p-1)N/(N+2-\alpha)$\ which we have excluded to guarantee that
$1/r_{1}<1/r_{0}$ in Lemma \ref{conv}\ above. Note however that $\Lambda
_{\alpha,p}\neq\emptyset$ iff $\Lambda\neq\emptyset$, because
$2(p-1)N/(N+2-\alpha)<2N/(N-2)$ if $p<(2N-\alpha)/(N-2).$

\begin{proposition}
\label{aa}Assume that $(H_{1})$ holds. Then every positive solution $u$ to
problem \emph{(\ref{lim})} is radially symmetric with respect to some point
and satisfies the following:\newline(i) If $p>2$ then%
\begin{equation}
u(x),\left\vert \nabla u(x)\right\vert =O(\left\vert x\right\vert
^{-\frac{N-1}{2}}e^{-\sqrt{\lambda}\left\vert x\right\vert })\text{\hspace
{0.2in}as }\left\vert x\right\vert \rightarrow\infty. \label{p>2}%
\end{equation}
(ii) If $p=2$ then, for every $\varepsilon\in(0,\lambda),$%
\begin{equation}
u(x),\left\vert \nabla u(x)\right\vert =O(\left\vert x\right\vert
^{-\frac{N-1}{2}}e^{-\sqrt{\lambda-\varepsilon}\left\vert x\right\vert
})\text{\hspace{0.2in}as }\left\vert x\right\vert \rightarrow\infty.
\label{p=2}%
\end{equation}

\end{proposition}

\begin{proof}
Let $u$ be a positive solution to problem (\ref{lim}). Ma and Zhao showed
that, if $(H_{1})$ holds, then $u$ is radially symmetric and monotone
decreasing with respect to some point \cite[Theorem 2]{mz}. Fix $\varepsilon
\in(0,\lambda).$ Lemma \ref{conv}\emph{(iii)} implies that there exists
$R_{1}>0$ such that $\left(  \left(  K\ast u^{p}\right)  u^{p-1}\right)
(x)<\varepsilon u(x)$ for every $\left\vert x\right\vert \geq R_{1}.$ It
follows that
\[
-\Delta u(x)+(\lambda-\varepsilon)u(x)<0\quad\text{for}\ \left\vert
x\right\vert \geq R_{1}.
\]
On the other hand, for $R_{2}$ large enough, the function $v(x):=\left\vert
x\right\vert ^{-\frac{N-1}{2}}e^{-\sqrt{\lambda-\varepsilon}\left\vert
x\right\vert }$ satisfies%
\[
-\Delta v(x)+(\lambda-\varepsilon)v(x)\geq0\quad\text{for}\ \left\vert
x\right\vert \geq R_{2}.
\]
Therefore, by the maximum principle,
\[
u(x)\leq\left\vert x\right\vert ^{-\frac{N-1}{2}}e^{-\sqrt{\lambda
-\varepsilon}\left\vert x\right\vert }\quad\text{for}\ \left\vert x\right\vert
\geq\max\{R_{1},R_{2}\}.
\]
This proves that $u$ satisfies (\ref{p=2}). Moreover, since $u(x)=O(e^{-\sqrt
{\lambda-\varepsilon}\left\vert x\right\vert })$ for all $\varepsilon
\in(0,\lambda),$ if $p>2$ there exists $\beta>\sqrt{\lambda}$ such that
\[
\left(  K\ast\left\vert u\right\vert ^{p}\right)  u^{p-1}(x)=O(e^{-\beta
\left\vert x\right\vert }).
\]
The same argument given in the last part of the proof of \cite[Proposition
1.4]{gnn} yields estimate (\ref{p>2}) for $u$. Using interior estimates we
obtain the decay estimates for $\left\vert \nabla u(x)\right\vert .$
\end{proof}

\end{document}